\documentclass[11pt]{article}

\usepackage{ucs}
\usepackage[utf8x]{inputenc}
\usepackage[hmargin=1in,vmargin=1.5in]{geometry}
\usepackage{amsmath}
\usepackage{amsfonts}
\usepackage{amssymb}
\usepackage{amsthm}
\usepackage{mathrsfs}
\usepackage{mathtools}
\usepackage{graphicx}
\usepackage{color}
\usepackage{vruler}

\usepackage{hyperref}

\theoremstyle{plain}
\newtheorem{theorem}{Theorem}
\newtheorem*{theorem*}{Theorem}
\newtheorem{lemma}[theorem]{Lemma}

\newtheorem{proposition}[theorem]{Proposition}

\theoremstyle{definition}

\theoremstyle{remark}

\author{Louigi Addario-Berry\thanks{Department of Mathematics and Statistics, McGill University, Montr\'eal, Canada. \texttt{louigi.addario@mcgill.ca}. Partially supported by an NSERC Discovery Grant and Discovery Accelerator Supplement and an FQRNT Team Research Grant. 
} \and Pascal Maillard\thanks{Laboratoire de Math\'ematiques d'Orsay, Univ.~Paris--Sud, CNRS, Universit\'e Paris--Saclay, 91405 Orsay Cedex, France. \texttt{pascal.maillard@u-psud.fr}. Partially supported by ANR Liouville (ANR-15-CE40-0013), ANR GRAAL (ANR-14-CE25-0014) and a CRM Simons Research Fellowship.}}
\title{The algorithmic hardness threshold for continuous \\ random energy models} 
\date{October 24, 2018}
\def\N{\mathbb N}

\def\R{\mathbb R}

\def\P{\mathbb P}
\def\E{\mathbb E}
\def\F{\mathscr F}
\DeclareMathOperator{\Var}{Var}
\def\Ind{\boldsymbol 1}
\def\eps{\varepsilon}

\DeclareMathOperator*{\argmax}{argmax}

\def\T{\mathbb T}

\newcommand\rX{\mathrm{X}}
\usepackage[normalem]{ulem}

\numberwithin{equation}{section}
\numberwithin{theorem}{section}

\begin{document}
\maketitle


\begin{abstract}
We prove an algorithmic hardness result for finding low-energy states in the so-called \emph{continuous random energy model (CREM)}, introduced by Bovier and Kurkova in 2004 as an extension of Derrida's \emph{generalized random energy model}. The CREM is a model of a random energy landscape $(X_v)_{v \in \{0,1\}^N}$ on the discrete hypercube with built-in hierarchical structure, and can be regarded as a toy model for strongly correlated random energy landscapes such as the family of $p$-spin models including the Sherrington--Kirkpatrick model. The CREM is parameterized by an increasing function $A:[0,1]\to[0,1]$, which encodes the correlations between states. 

We exhibit an \emph{algorithmic hardness threshold} $x_*$, which is explicit in terms of $A$. More precisely, we obtain two results: First, we show that a renormalization procedure combined with a greedy search yields for any $\varepsilon > 0$ a linear-time algorithm which finds states $v \in \{0,1\}^N$ with $X_v \ge (x_*-\varepsilon) N$. Second, we show that the value $x_*$ is essentially best-possible: for any $\varepsilon > 0$, any algorithm which finds states $v$ with $X_v \ge (x_*+\varepsilon)N$ requires exponentially many queries in expectation and with high probability. We further discuss what insights this study yields for understanding algorithmic hardness thresholds for random instances of combinatorial optimization problems.


\end{abstract}


\section{Introduction}
Write $\T_N$ for the rooted binary tree of depth $N$. Nodes at depth $i$ are indexed by strings $v_1v_2\ldots v_i \in V_i:=\{0,1\}^i$. We use $\emptyset$ to denote the root. For $v,w \in \T_N$ we denote by $v\wedge w$ the most recent common ancestor of $v$ and $w$; if $v=w$ then $v \wedge w=v$. We also write $v\le w$ and $v<w$ to mean that $v$ is an ancestor of $w$, including or not including $w$. The generation of node $v$ is denoted $|v|$. 

Let $A$ be the distribution function of an arbitrary probability distribution on $[0,1]$, and define a  centred Gaussian process $\mathrm{X}=(X_v)_{v \in \T_N}$ by 
\[
\E[X_v X_{w}] = N A(R_N(v,w)),
\]
where 
\[
 R_N(v,w) = \frac 1 N \max\{i: v_j = w_j\ \forall j\le i\} = \frac{|v \wedge w|}{N}. 
\]
One may view $\mathrm{X}$ as a time-inhomogeneous binary branching random walk in which the offspring of generation-$i$ particles have independent centred Gaussian increments with variance 
\[
N(A((i+1)/N) -A(i/N)).
\]
In particular, $\mathrm{X}$ has the following \emph{branching property}: For every vertex $v\in \T_N$, the family $(X_w - X_v)_{w \ge v}$ is independent of the family $(X_w)_{w\not > v}$

The {\em continuous random energy model} (or {\em CREM}) with parameters $A$ and $N$ is the Gaussian process $(X_v)_{v \in V_N}$ obtained from $\mathrm{X}$ by only considering generation-$N$ nodes. The CREM is an extension of Derrida's \emph{generalized random energy model} \cite{Derrida1985} was introduced in \cite{Bovier2004a} as an analytically tractable model of mean-field spin glasses. That work described the limiting free energy, ground state, and overlap distribution of Gibbs measures on the CREM. The contribution of this work is to explicitly describe the algorithmic threshold for finding low-energy states in the CREM (equivalently, finding nodes $v \in V_N$ for which $X_v$ is large). 

We use the following natural computational model of (randomized) algorithms. We first extend our probability space by adding a sequence $U_1,U_2,\ldots$ of iid Uniform$[0,1]$ random variables, independent of the CREM. Given a sequence $\mathrm{v}=(v(n))_{n\ge1}$ of random vertices of $\T_N$, we define a filtration $\F = (\F_n)_{n\ge0}$ where $\F_n$ contains ``all information about everything we have queried so far, as well as the additional randomness needed to choose the next vertex'', i.e.
\[
\F_n = \sigma\left(v(1),\ldots,v(n);\,X_{v(1)},\ldots,X_{v(n)};\,U_1,\ldots,U_{n+1}\right).
\]
We say the sequence $\mathrm{v}$ is a \emph{randomized search algorithm on $\mathbb{T}_N$} (algorithm for short) if it is previsible with respect to the filtration $\F=\F(\mathrm{v})$, i.e., if $v(n+1)$ is $\F_n$-measurable for every $n\ge0$.

\begin{theorem}
\label{th:complexity_CREM}
Suppose that there exists a Riemann-integrable function $a:[0,1]\to\R_+$, such that $A(t) = \int_0^t a(s)\,ds$ for all $t\in[0,1]$. 
Then, letting 
\[
x_*=x_*(A)=\sqrt{2\log 2}\int_0^1 \sqrt{a(t)}dt\,,
\] 
the following holds. 
\begin{enumerate}
 \item For all $x < x_*$, there is a linear-time algorithm that finds a node $v\in V_N$ with $X_v \ge x N$ with high probability. 
 \item For all $x > x_*$, there exists $\gamma=\gamma(A,x)>0$ such that for $N$ sufficiently large, for any algorithm, the number of queries performed before finding a node $v\in V_N$ with $X_v \ge x N$ is stochastically bounded from below by a geometric random variable with parameter $\exp(-\gamma N)$.
\end{enumerate}
\end{theorem}
Note in particular that the Riemann-integrability assumption in the statement of the theorem implies that the function $a$, defined therein, is bounded, and that the function $t\mapsto \sqrt{a(t)}$ is also Riemann-integrable; both facts are used below. 

The first assertion of the theorem should be interpreted as follows: for all $x < x_*$, there is $C=C(x)$ such that for all $\eps > 0$, for all sufficiently large $N$ there is an algorithm $\mathrm{v}$ and a stopping time $\tau$ for the filtration $\F(\mathrm{v})$ with 
\[
\P(\tau \le C N,v(\tau) \in V_N)=1\quad\mbox{and with}\quad \P(X_{v(\tau)} \ge xN) > 1-\eps\,.
\]
It is impossible to strengthen this to have $\P(X_{v(\tau)} \ge xN)=1$, since $\P(\sup_{v \in V_N} X_v < xN)>0$. The second assertion states that for any algorithm $\mathrm{v}$, the stopping time 
\[
\tau_x = \tau_x(\mathrm{v},N) = \inf\{n\in\N: v(n)\in V_N\text{ and } X_{v(n)} \ge xN\},
\] 
stochastically dominates a Geometric$(\exp(-\gamma N))$ random variable. 

\subsection{Discussion - how to understand $x_*$.}\label{sec:howtounderstand}

For the remainder of the paper we fix functions $A$ and $a$ satisfying the conditions of Theorem~\ref{th:complexity_CREM}. Denote by $\hat A$ the concave hull of $A$ and by $\hat{a}$ its left-derivative. Then the (negative) ground state energy satisfies 
\begin{equation}\label{eq:xsdef}
 \frac 1 N \sup_{v\in V_N} X_v  \stackrel{\mathrm{a.s.}}{\longrightarrow} \sqrt{2\log 2}\int_0^1 \sqrt{\hat{a}(s)}\,ds \eqqcolon x_s.
\end{equation}
This is proved in \cite[Theorem~3.1]{Bovier2004a} -- the authors assume in that paper that $A$ is continuously differentiable, but inspection of the proof shows that this is not needed for that theorem. 

The quantity $x_s$ admits a representation in terms of a variational problem which we now describe (this can be deduced from results of Mallein \cite{Mallein2015}; see Appendix~\ref{sec:app} for details). 
For $b:[0,1] \to \R$ measurable and $t \in [0,1]$, let 
\[
E(b,t) = -(\log 2) t + \int_0^t \frac{b(s)^2}{2a(s)}\,ds,
\]
where we set $0/0=0$. Then define 
\[
 \mathcal Z = \{z:[0,1] \to \mathbb{R}\mbox{ absolutely continuous}:z'(0) = 0\mbox{ and } \forall t\in[0,1], E(z',t) \le 0\}.
\]
Heuristically, $\mathcal Z$ is the closure of the set of all ``energetically admissible'' macroscopic particle trajectories along branches of $\mathbb{T}_N$: if $z$ has $E(z',t) < 0$ for all $t \in (0,1]$ then for all $M \le N$, the expected number of particles $v \in V_M$ with 
\[
(N^{-1}X_{v(\lfloor t N \rfloor)},0 \le t \le M/N) \approx (z_t,0 \le t \le M/N) 
\]
is exponentially large in $N$. This criterion is in fact sufficient for the existence with high probability of exponentially many particles approximately following the trajectory $z$ until the terminal time $N$ (see \cite[Section 1.2]{Mallein2015}). It is clear from this description that the ground state energy satisfies 
\begin{equation}
\label{eq:xm_var_problem}
x_s =  \sup\{z(1): z \in \mathcal{Z}\}.
\end{equation}
The supremum in (\ref{eq:xm_var_problem}) is achieved by the function $z:[0,1] \to [0,\infty)$ given by 
\begin{equation}
\label{eq:z}
z(t) = \int_0^t a(s) \cdot \Big(\frac{2\log 2}{\hat{a}(s)}\Big)^{1/2}\mathrm{d}s\, ;
\end{equation}
see Proposition~\ref{prop:r_max} for a proof that $z \in \mathcal{Z}$ and that $x_s=z(1) = (2\log 2)^{1/2} \int_0^1 \hat{a}(s)^{1/2}$. 

Note that from the definition of $x_s$ in (\ref{eq:xsdef}), by the Cauchy--Schwarz inequality,
\[
 \frac{x_s}{\sqrt{2\log 2}} \le \left(\int_0^1 \hat{a}(s)\right)^{1/2} = (\hat A(1) - \hat A(0))^{1/2} = 1.
\]
In the above, equality holds if and only if $\hat{a}$ is the constant function. This is equivalent to the condition that $\hat A(t) = t$, which in turn is equivalent to the condition that $A(t) \le t$ for all $t\in[0,1]$. This is called the \emph{weak correlation regime} in \cite{Bovier2004a}, and includes the usual branching random walk, for which $A(t) = t$ for all $t\in[0,1]$. 

Finally, we define a special trajectory $z_*$ by
\begin{equation}
\label{eq:zstar}
 z_*(t) = \sqrt{2\log 2} \int_0^t \sqrt{a(s)}\,ds\, ;
\end{equation}
this is called the {\em natural speed path} in \cite{Mallein2015}. 
Note that $z_*$ satisfies $E(z_*,t) = 0$ for all $t\in[0,1]$ and so $z_*\in \mathcal Z$. 
We then have 
\[
x_* := z_*(1) =  \sqrt{2\log 2} \int_0^1 \sqrt{a(s)}\,ds, 
\]
and when $x < x_*$ and $x_*-x$ is small, the linear time algorithms promised in Theorem~\ref{th:complexity_CREM} in fact find paths in $\mathbb{T}_N$ from the root to $V_N$ which approximately follow the curve $z_*$. Conversely, if $x>x_*$, then any path leading to $x$ has to have a steeper ascent than the natural speed path at some time. The fact that this is exponentially unlikely for any fixed path eventually leads to the algorithmic hardness result.

Finally, comparing $z$ and $z_*$ from equations \eqref{eq:z} and \eqref{eq:zstar} and using the Cauchy--Schwarz inequality, it is also not hard to see that $x_* \le x_s$, with equality if and only if $A$ is concave. {\em The concavity of $A$ is therefore a necessary and sufficient condition for the existence of efficient algorithms for finding a node whose energy is within a $(1-o(1))$ factor of the ground state energy.}

\subsection{Discussion - phase transitions, algorithms, and spin glasses}\label{sec:discussion}

This work is motivated by the large (and growing) body of work on random energy landscapes, and the link between the structure of such landscapes and algorithmic barriers. They arise in areas such as \emph{combinatorial optimization}, \emph{spin glasses} and \emph{statistical inference}. In this section, we briefly put our work into the context of the first two areas. 
For the relation to statistical inference problems, we refer to the recent surveys \cite{MR3699594,Zdeborova2016}.

In a \emph{combinatorial optimization problem}, one is given a combinatorial structure and one seeks to find an element from that structure which is in some sense large or maximal. A prime example is the \emph{independent set problem}, in which one is interested in finding an independent set of a given density $\alpha$ in a given input graph. Such problems being typically NP-complete and therefore (conjecturally) hard to solve for arbitrary inputs, in the last decades there has been a lot of interest in the efficiency of algorithms on \emph{random} instances of such problems. A major thread within this research attempts to link the geometry of the space of solutions of random instances with algorithmic barriers. In particular, several \emph{thresholds} or \emph{phase transitions} in the geometry of the solution space have been proposed, together with interpretations in terms of algorithmic efficiency \cite{Krzakala2007}. We focus here on the so-called \emph{condensation threshold} $\alpha_c$. 


In the aforementioned independent set problem, one can for example consider a random $d$-regular graph on $N$ vertices for fixed $d$ and large $N$, and ask whether one can, with high probability, efficiently find an independent set of size at least $\alpha N$, for some $\alpha \in(0,1)$.
In this context, roughly speaking, for $\alpha$ just above the condensation threshold $\alpha_c$ the uniform measure on independent sets of density $\alpha$ is expected to ``condense'' onto an $O(1)$ number of ``clusters'' of comparable size. In contrast, for $\alpha$ below this threshold, such a measure is either comprised of an exponential in $N$ number of clusters or of a single big cluster. 

To this date, the meaning of the condensation threshold in terms of algorithmic efficiency seems to be unclear. It was conjectured \cite{Krzakala2007} that it is the threshold up to which algorithms inspired by a certain iterative procedure called \emph{belief propagation} should be effective, but recent rigorous results \cite{Coja-Oghlan2011,Coja-Oghlan2017} suggest that this might not be the case. In what follows, and on the basis of Theorem~\ref{th:complexity_CREM}, we will indeed argue that \emph{one should not expect any relation between the condensation threshold and algorithmic efficiency.}


The insight on phase transitions in combinatorial optimization problems draws a lot on the understanding of disordered systems in physics, in particular of \emph{(infinite-range) spin glasses}. A prime example is the \emph{Ising $p$-spin model} (with zero magnetization), $p\ge 2$, which is the random probability distribution on the hypercube $\{-1,1\}^N$ defined by \cite{Derrida1981}
\[
\mu_\beta(\sigma) \propto \exp\left(\frac{\beta}{N^{(p-1)/2}} \sum_{i_1,\ldots,i_p=1}^N J_{i_1,\ldots,i_p}\sigma(i_1)\cdots \sigma(i_p)\right),
\]
where $J_{i_1,\ldots,i_p}$ are iid standard Gaussian random variables and $\beta\ge0$ is a parameter called the \emph{inverse temperature}. The term inside the exponential (without the factor $\beta$), is also called the \emph{negative energy} of the configuration $\sigma$. With $p=2$, this model is also called the \emph{Sherrington--Kirkpatrick (SK) model}, and corresponds to the combinatorial {\em max-cut} problem with Gaussian weights. A continuous version, where the hypercube is replaced by the $(N-1)$-dimensional sphere of radius $\sqrt N$, is called the \emph{spherical $p$-spin model}. These models have been studied in depth since the 70's and a deep understanding of several aspects has been obtained by both rigorous \cite{PanchenkoBook,TalagrandI,TalagrandII} and non-rigorous \cite{MePaVi} methods. In particular, it is now well understood that spin glasses undergo several phase transitions in the parameter $\beta$. Of importance to us is the \emph{critical inverse temperature} $\beta_c$, which marks the point where the model goes from the so-called \emph{paramagnetic} ($\beta<\beta_c$) to the \emph{spin glass} phase ($\beta > \beta_c$).\footnote{Mathematically, one can understand this transition as follows: when sampling two independent configurations $\sigma,\sigma'$ from the (random) measure $\mu_\beta$, then their renormalized scalar product $\frac 1 N \sigma\cdot \sigma'$ converges in law to a constant in the paramagnetic phase, whereas it converges to a non-degenerate random variable in the spin glass phase. This definition also holds for the CREM considered in this paper, with the scalar product $\sigma\cdot\sigma'$ replaced by $|v\wedge w|$.} We don't explain here the precise physical meaning of this sentence --- the interested reader can find this for example in \cite{MePaVi,SteinNewmanBook}. The relevance for us is that $\beta_c$ corresponds to the condensation threshold $\alpha_c$ defined above \cite{Krzakala2007}.

In spin glass models, the parameter $\beta$ is in correspondence with a \emph{negative energy $x = x(\beta)$}, in the sense that a sample from the measure $\mu_\beta$ has with high probability energy approximately $-xN$. As a consequence, the threshold $\beta_c$ corresponds to a threshold $x_c$ in the negative energy. In the CREM considered here, the values of both thresholds are in fact the same \cite[Theorem 3.3]{Bovier2004a}:
\[
x_c = \beta_c = \frac{\sqrt{2\log 2}}{\lim_{t\to 0} \sqrt{\hat a(t)}} = \frac{\sqrt{2\log 2}}{\sup_{t\in(0,1]} \sqrt{\hat a(t)}}
\]

The following fact is a consequence of Theorem~\ref{th:complexity_CREM} and deserves, in our opinion, particular attention: \emph{both $x_*<x_c$ and $x_* > x_c$ are possible.} Here are examples which illustrate this fact. First, suppose $A(x) < x$ for all $x\in(0,1)$. Then $x_c = x_s = \sqrt{2\log 2}$, but since $A\ne \hat A$, we have $x_* < x_s = x_c$. Second, suppose that $A$ is strictly concave. Then $x_c < x_s$, since
\[
x_c = \sqrt{2\log 2} \int_0^1\frac{\hat a(s)}{\sup_{t\in(0,1]} \sqrt{\hat a(t)}}\,ds < \sqrt{2\log 2} \int_0^1 \sqrt{\hat a(s)}\,ds = x_s.
\]
But since $A = \hat A$, we have $x_* = x_s > x_c$. Of course, one can cook up examples which sit in between these two extremes.

To conclude, we have compared our algorithmic threshold $x_*$ from Theorem~\ref{th:complexity_CREM} to the threshold $x_c$ from the spin glass literature corresponding to the critical inverse temperature $\beta_c$ and have seen that $x_*$ does not have any relation to $x_c$ in general. 
We believe this insight is also valid for combinatorial optimization problems. To be precise, for a given combinatorial optimization problem, let $\alpha^*$ be the threshold above which it becomes impossible for \emph{any} algorithm to find solutions in polynomial time with high probability. Then our results suggest that both $\alpha^*<\alpha_c$ and $\alpha^* > \alpha_c$ are possible, with $\alpha_c$ the condensation threshold.

The inequality $\alpha^* > \alpha_c$ appears to hold for the SK model/Max-cut with Gaussian weights. In fact in this setting simulations suggest that sequential improvement algorithms find configurations lying well beyond the condensation threshold \cite{Eastham2006}. 
Relatedly, for Max-cut on random regular graphs, Montanari \cite{Montanari2016} has conjectured that local algorithms find asymptotically optimal configurations, which would imply that $\alpha^*=\alpha_s > \alpha_c$ for this model. (This contrasts with the situation for a generalization of Max-cut to random hypergraphs, where local algorithms are provably suboptimal in some cases \cite{Chen2017b}.) On the other hand, we are not aware of an example where it is known or concretely conjectured that $\alpha^*<\alpha_c$, but our results suggest that this situation should be the rule rather than the exception for models exhibiting what in spin-glass language is called {\em one-step replica symmetry breaking} at zero temperature (which for the CREM is the case exactly in the weak correlation regime $A(x)\le x$ for all $x\in[0,1]$).

We briefly mention another threshold that has been proposed to be of importance for combinatorial optimization problems: the \emph{clustering} or \emph{shattering threshold} $\alpha_d \le \alpha_c$ at which the uniform measure on the space of solutions ``shatters into clusters'' \cite{Achlioptas2006,Krzakala2007}. Related phenomena happen at the point of \emph{dynamical arrest} or \emph{dynamical glass transition} $\beta_d\le \beta_c$ in the spin glass and glass literature \cite{Montanari2004}. In that context, the threshold $\beta_d$ corresponds to the point at which a natural spin-flip dynamics with stationary measure $\mu_\beta$, also known as \emph{Glauber dynamics}, undergoes a transition from fast ($\beta<\beta_d$) to slow ($\beta>\beta_d$) mixing. Unfortunately, we are not aware of a natural Glauber dynamics for the CREM, and could not find a natural definition for a threshold $x_d$ corresponding to $\beta_d$ in our setting. However, it is reasonable to conjecture for general spin glass models the existence of a threshold $\beta_G$, such that one can efficiently approximate the Gibbs measure $\mu_\beta$ if $\beta < \beta_G$ and cannot if $\beta > \beta_G$. Of course, we have $\beta_G \ge \beta_d$, because the Glauber dynamics efficiently approximate the Gibbs measure if $\beta < \beta_d$. Moreover, this threshold, as well as its obvious counterpart $x_G$, is well-defined for the CREM. Note that $x_G\le x_*$, because sampling a typical vertex of negative energy $x_GN$ is clearly harder than finding \emph{some} such vertex. In Section~\ref{sec:open_questions}, we give a precise definition of $x_G$ and conjecture its value for the CREM. In particular, we find that one may again have $x_G > x_c$ or $x_G < x_c$. 

In the future, we hope that one can use the insights from the current article to make more quantitative predictions for algorithmic hardness thresholds in other models, or even to devise new algorithms with provable efficiency. In a fairly broad range of spin glasses, for $N$ large, independent samples from $\mu_\beta$ (with the same underlying random environment $J$) are expected to approximately satisfy {\em ultrametricity}, meaning that they are in a certain sense organized according to an underlying tree structure. (A precise definition of approximate ultrametricity would lead us too far afield; a nice formulation can be found in \cite{MR3628881}.) It is very natural to ask for which models it is the case that, if one can discover such a tree structure, one can exploit it to describe the algorithmic threshold for finding low energy states. In the CREM, a tree structure is built into the model, which makes it in some sense a natural proving ground for this question. 

We finally remark that our results are in line with the conjectured qualitative behavior of  algorithmic hardness thresholds for the $p$-spin glass models. Namely, for spherical models there are indications \cite{Auffinger2013a} that this threshold differs from the ground state when $p\ge 3$, whereas it is trivially equal to the ground state for $p=2$; see Open Question~\ref{item:spherical_threshold} in Section~\ref{sec:open_questions} for further details. One might expect the same for the Ising $p$-spin glass models; when $p=2$ this is related to the aforementioned conjecture of Montanari regarding local algorithms for Max-cut. Now, following Derrida~\cite{Derrida1985}, one can model the Ising $p$-spin glass by a CREM with parameter $A(t) = (I^{-1}((\log 2) t))^p$, where $I(x) = [(1+x) \log(1+x) + (1-x)\log (1-x)]/2$ for $x\in[0,1]$ and $I^{-1}$ is its inverse. (Bovier and Kurkova \cite{Bovier2004a} provide a discussion of which features of Ising spin glasses are captured by such a CREM.) One easily checks that this function is concave, and so $x_* = x_s$ by Theorem~\ref{th:complexity_CREM}, if and only if $p\le 2$, supporting the above picture.

\medskip
\noindent{\bf Overview.}
The remainder of the work is structured as follows. In Section~\ref{sec:existence} we prove the first part of Theorem~\ref{th:complexity_CREM}, showing the existence of linear time algorithms when $x < x_*$. In Section~\ref{sec:hardness}, we prove the second part of Theorem~\ref{th:complexity_CREM}, showing exponential hardness when $x > x_*$. 
Section~\ref{sec:weakermodel} sketches how to modify the algorithm if one only has access to the values $(X_v,v \in V_N)$ and not to values at internal nodes of $\mathbb{T}_N$; the cost is a linear-time slowdown, resulting in an algorithm whose runtime is $O(N^2)$ in probability. Section~\ref{sec:open_questions} contains open questions and potential avenues for continuing the line of research initiated here. 
Finally, Appendix~\ref{sec:app} puts the discussion in Section~\ref{sec:howtounderstand} on rigorous footing.

\medskip
\noindent{\bf Acknowledgements.}
We thank Gerard Ben Arous, Alexander Fribergh, Andrea Montanari and Jean-Christophe Mourrat for stimulating and enlightening discussions regarding spin glasses and their algorithmic complexity. We also warmly thank Bastien Mallein for suggesting formula~\eqref{eq:bastien}.

\section{Existence of a linear-time algorithm for $x<x_*$} \label{sec:existence}

Fix $M\in\N$ large. We define the following algorithm.
\begin{itemize}
\item Let $v(M,0)$ be the root of $\T_N$. 
\item For $0 \le n < \lceil N/M \rceil$, inductively define 
\[
v(M,n+1) = \argmax_{\substack{v\ge v(M,n)\\|v| = (Mn+M) \wedge N}} X_v.
\]
\item Set $v_{\mathrm{out}}(M) = v(M,\lceil N/M \rceil)$. 
\end{itemize}

This algorithm has running time bounded by $1+2^M \lceil N/M\rceil$, and so is a linear-time algorithm for fixed $M$. It remains to show that we can choose $M$ such that with high probability, $X_{v_{\mathrm{out}}(M)} \ge x\cdot N$. This will follow from the following lemma:

\begin{lemma}
\label{lem:exp_var}
For every $x' < x_*$, there exists $M=M(x')\in \N$ such that for all $N$ sufficiently large, $\E[X_{v_{\mathrm{out}}(M)}] \ge x'\cdot N$ and $\Var(X_{v_{\mathrm{out}}(M)}) \le N$.
\end{lemma}

We first show how the first part of Theorem~\ref{th:complexity_CREM} follows from the above.

\begin{proof}[Proof of first part of Theorem~\ref{th:complexity_CREM}]
Fix $x' \in (x,x_*)$ and let $M=M(x')$ be as in Lemma~\ref{lem:exp_var}. Then 
for sufficiently large $N$, we have $\E[X_{v_{\mathrm{out}}(M)}] \ge x'\cdot N$ and $\Var(X_{v_{\mathrm{out}}(M)}) \le N$. Chebychev's inequality now gives that $\E[X_{v_{\mathrm{out}}(M)}] \ge xN$ with high probability as $N\to\infty$. Since $M$ is fixed, as mentioned above, the algorithm for finding $v_{\mathrm{out}}(M)$ has running time linear in $N$; the first part of the theorem follows.
\end{proof}

We now prove the missing piece.

\begin{proof}[Proof of Lemma~\ref{lem:exp_var}]
Write $Y_n = X_{v(n)}$. By the branching property, $(Y_n - Y_{n-1})_{n\ge1}$ is an independent sequence of random variables and so their variances (and of course their expectations) sum up. The variance bound is easiest: By a standard application of the Poincar\'e inequality for Gaussian measures (see e.g. Exercise 3.24 of Boucheron--Lugosi--Massart), we have for every $M$,
\begin{align*}
\Var(Y_n - Y_{n-1}) &= \Var(Y_n - Y_{n-1}\,|\,X_{v(n-1)},\,v(n-1))\qquad \text{(by the branching property)}\\
&\le \max_{\substack{v\ge v(n-1)\\ |v| = (|v(n-1)| + M) \wedge N}} \Var(X_v - X_{v(n-1)}\,|\,X_{v(n-1)},\,v(n-1))\\
& = (A(nM/N\wedge 1) - A((n-1)M/N\wedge 1))\cdot N.
\end{align*}
Summing over $n$, this gives for every $M$,
\[
\Var(X_{v_{\mathrm{out}}}) \le (A(1) - A(0))\cdot N = N,
\]
which proves the variance bound. 

The proof of the expectation estimate is a little more involved. Define for $\delta >0$ and $1\le n\le \lceil 1/\delta\rceil$,
\[
\alpha(n,\delta)
= \min_{s\in [\delta (n-1),\delta n]} a(s),
\]
where we take $a(s) = 0$ for $s> 1$. By the Riemann-integrability of $\sqrt a$, we have
\begin{equation}
\label{eq:adelta_limit}
\lim_{\delta \to 0} \delta \sum_{n=1}^{\lfloor 1/\delta\rfloor} \sqrt{\alpha(n,\delta)} = \int_0^1 \sqrt {a(t)}\,dt.
\end{equation}
We claim that for every $\eps>0$, we can choose $M$ such that for large enough $N$, for all $1\le n \le N/M$,
\begin{equation}
\label{eq:Y}
\E[Y_n - Y_{n-1}] \ge (1-\eps) \sqrt{2\log 2} \sqrt{\alpha(n,M/N)}\times M.
\end{equation}
We briefly postpone the proof of this bound. Next, it is a classical consequence of the Sudakov-Fernique inequality that
\begin{equation}
\label{eq:last_piece}
\E[Y_{\lceil N/M\rceil} - Y_{\lceil N/M\rceil-1}] \ge 0.
\end{equation}
It follows from \eqref{eq:adelta_limit}, \eqref{eq:Y} and \eqref{eq:last_piece} that for large $N$,
\begin{align*}
\E[X_{v_{\mathrm{out}}}] &= \sum_{n=1}^{\lfloor N/M\rfloor + 1} \E[Y_n - Y_{n-1}] \\
&\ge (1-2\eps) \sqrt{2\log 2} \left(\int_0^1 \sqrt {a(t)}\,dt\right) N\\
&= (1-2\eps) x_* N.
\end{align*}
This proves the expectation bound from the statement of the lemma.

It remains to prove \eqref{eq:Y}. For simplicity of notation, we only do so for $n=1$; the general case is no different. Fix $v,w\in V_M$; then 
\begin{align*}
\E[(X_v - X_w)^2] 
&= 2(A(M/N) - A(|v\wedge w|/N))N \\
&\ge 2 \alpha(1,M/N) (M-|v\wedge w|)\\
&= \alpha(1,M/N) E[(X'_v - X'_w)^2],
\end{align*}
where $X'$ is the binary branching random walk whose offspring displacement is standard Gaussian. 
By the Sudakov-Fernique inequality, we get
\[
\E[Y_1] = \E\left[\max_{|v|=M} X_v\right] \ge \sqrt{\alpha(1,M/N)} \ \E\left[\max_{|v|=M} X'_v\right].
\]
Also, by the Hammersly-Kingman-Biggins theorem \cite{MR0420890,MR0370721,MR0400438}, the branching random walk $X'$ satisfies
\[
 \frac{\E\left[\max_{|v|=M} X'_v\right]}{M} \to \sqrt{2\log 2},\quad M\to\infty.
\]
Choosing $M$ large enough so that the term on the left-hand side of the last display is greater than $(1-\eps)\sqrt{2\log 2}$, this gives \eqref{eq:Y} (for $n=1$) and concludes the lemma.
\end{proof}

\section{Exponential hardness for $x>x_*$}\label{sec:hardness}

In this section, we prove the second part of Theorem~\ref{th:complexity_CREM}. 
The first observation is a completely deterministic fact on the particle trajectories. For $v\in \T_N$ and $0 \le m \le |v|$, let $v[m] = v_1\ldots v_{m}$ be the generation-$m$ ancestor of $v$. For $\eps > 0$ and $K \in \N$, say that $v$ is {\em $(\eps,K,\mathrm{X})$-steep} if there exists $k \in \{1,\ldots,K\}$ such that $|v|=\lfloor Nk/K \rfloor$ and such that 
\begin{equation}
\label{eq:subproblem}
 X_{v} - X_{v[\lfloor N(k-1)/K\rfloor]} \ge N\cdot \sqrt{(1 + \eps)\frac{(2\log 2) (A(k/K)-A((k-1)/K))}{K}}. 
\end{equation}
\begin{lemma}
\label{lem:deterministic}
 For all $x > x_*$, there exist $K\in\N$ and $\eps > 0$ such that 
 for all $N$ sufficiently large, 
all vertices $v\in V_N$ with $X_v \ge xN$ have an $(\eps,K,\mathrm{X})-steep$ ancestor. 
\end{lemma}

The proof of Lemma~\ref{lem:deterministic} can be found in Section~\ref{sec:proofs}. The heuristic behind it is the following: Recall that the trajectory $z_*$ satisfies $E(z_*,t) = 0$ for all $t\in[0,1]$. If a trajectory $z$ satisfies $(d/dt) E(z,t) \le 0$ for all $t\in[0,1]$, then $z(t)\le z_*(t)$ for all $t\in[0,1]$ and in particular, $z(1) \le x_*$. It follows that a trajectory $z$ ending at $z(1) = x > x_*$ must at some point exhibit an energetic increase: there must be $t \in [0,1]$ such that $z'(t) \ge \sqrt{(1+\eps)2\log 2 \times a(t)}$. The lemma is essentially a discretized version of this fact.

\begin{figure}
\begin{center}
\includegraphics{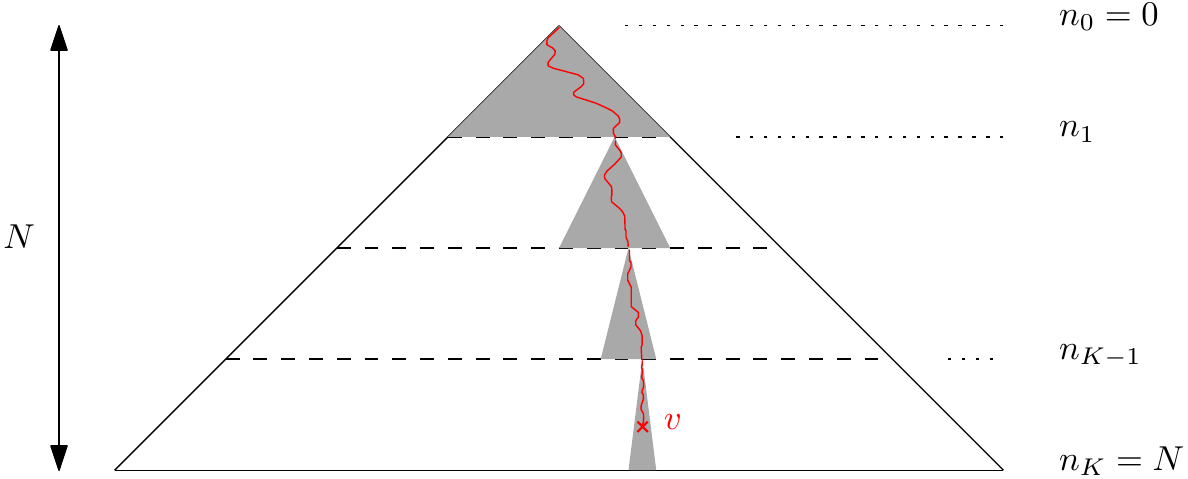}
\end{center}
\caption{\label{fig:chain_of_spindles} Illustration of the \emph{chain of spindles over $v$} (see main text for definition and notation).}
\end{figure}

We now describe how this is put to fruitful use for the proof of the hardness result. We will transform the search algorithm in two steps:
\begin{itemize}
\item We first localize the problem: Lemma~\ref{lem:deterministic} says that finding a vertex $v\in V_N$ with $X_v \ge xN$ is at least as hard as finding a vertex $v$ such that the inequality \eqref{eq:subproblem} is satisfied for some $k=1,\ldots,K$. We call this easier task the \emph{subproblem}.
\item We then eliminate dependencies by ``discovering'' more than just the value of the vertex we visit at each step. Specifically, for $v\in\T_N$, 
define the subset $\mathcal C_v\subset \T_N$ as follows:
\begin{align*}
\mathcal C_v &= \bigcup_{k=1}^K \mathcal C_{v,k},\quad\text{where}\\
\mathcal C_{v,k} &= \{w \in \T_N: |w| \le \lfloor Nk/K\rfloor \mbox{ and } \lfloor N(k-1)/K\rfloor \le |w \wedge v| \le \lfloor Nk/K\rfloor \}.
\end{align*}
\end{itemize}
Note that $\mathcal C_{v,k}$ is nonempty if and only if $|v| \ge \lfloor N(k-1)/K\rfloor$. 
We call $\mathcal C_v$ the \emph{chain of spindles over $v$}. See Figure~\ref{fig:chain_of_spindles} for an illustration. 



\begin{figure}
\begin{center}
\includegraphics{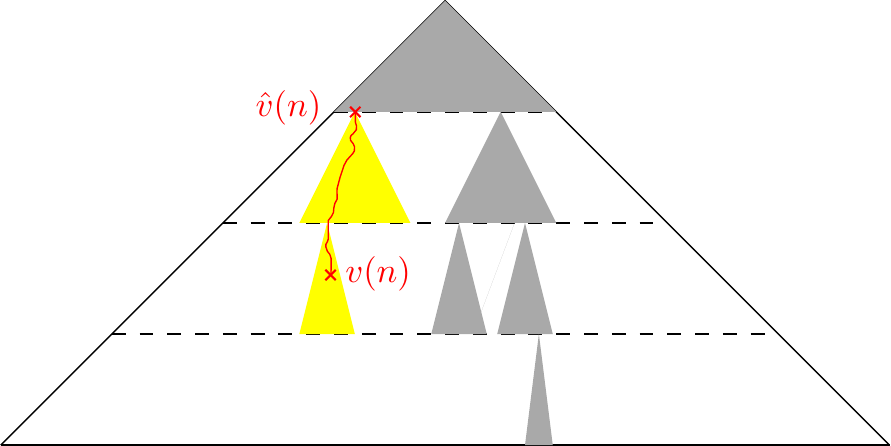}
\end{center}
\caption{\label{fig:independence} Illustration of the independence statement of Lemma~\ref{lem:independence}. The set $\mathcal C_{v(1)}\cup \ldots \cup \mathcal C_{v(n-1)}$ is shown in gray, the set $\mathcal C_{v(n)} \setminus (\mathcal C_{v(1)}\cup \ldots \cup \mathcal C_{v(n-1)})$ in yellow. The lemma implies in particular that the field $X$ on the yellow set is, up to a shift, independent of $\F_{n}$.}
\end{figure}

It is intuitively clear that we can only ``help'' the algorithm if at the $n$-th step we ``discover'' not only $X_{v(n)}$, but also the values of the vertices in the chain of spindles $\mathcal C_v$ over $v$. 
The advantage of this is that we create independence, which, as we will see below, allows us to bound the time for solving the subproblem from below by an geometrically distributed random variable. 

Formally, we define an enlarged filtration $\F'(\mathrm{v})=(\mathcal{F}_n')_{n \ge 0}$ by 
\[
\F'_n = \sigma\left(v(1),\ldots,v(n);\,X_{v},\, v\in \mathcal C_{v(1)}\cup \ldots \cup \mathcal C_{v(n)};\,U_1,\ldots,U_{n+1}\right) \supset \F_n.
\]
The algorithm $(v(n))_{n\ge1}$ is still previsible with respect to this filtration.
The next definitions are illustrated in Figure~\ref{fig:independence}. 
For $n \ge 1$, let $\mathcal{R}_n = \bigcup_{i \le n} \mathcal{C}_{v(i)}$ be the set of nodes in the chains of spindles over $v(1),\ldots,v(n)$. 
Also, let $\hat{v}(n)$ be the most recent ancestor of $v(n)$ in $\mathcal R_{n-1}$ if $n > 1$, and let $\hat{v}(n)$ be the root of $\mathbb{T}_n$ if $n=1$.

The following lemma formalizes the independence property mentioned above. 
\begin{lemma}
\label{lem:independence}
Fix any randomized search algorithm $\mathrm{v}=(v(n))_{n \ge 1}$. Then conditioned on $\F_{n-1}'$, the family of random variables $(X_v-X_{\hat{v}(n)})_{v\in \mathcal{R}_n\setminus \mathcal{R}_{n-1} }$ has the same law as $(X'_v-X'_{\hat{v}(n)})_{v\in \mathcal{R}_n\setminus \mathcal{R}_{n-1}}$, where $(X'_v)_{v \in \T_n}$ has the same law as $(X_v)_{v \in \T_n}$ and is independent of $\F_{n-1}'$. 
\end{lemma}
\begin{proof}
This is an immediate consequence of the branching property.
\end{proof}
The next lemma is the last ingredient for the proof of Theorem~\ref{th:complexity_CREM}. Its proof is again found in Section~\ref{sec:proofs}.
\begin{lemma}
\label{lem:subproblem}
For all $K\in\N$ and $\eps>0$, for any $\gamma \in (0,(\eps \log 2)/K)$, for all $N$ sufficiently large, for any $w\in\T_N$, 
\[
\P\left(\exists~v\in\mathcal C_{w}~:~v\mbox{ is }\mbox{$(\eps,K,\mathrm{X})$-steep}\right) \le e^{-\gamma N}
\] 
\end{lemma}
We are now ready to finish the proof of Theorem~\ref{th:complexity_CREM}.
\begin{proof}[Proof of the second part of Theorem~\ref{th:complexity_CREM}]
Let $x > x_*$ and let $(v(n))_{n\ge0}$ be a randomized search algorithm. As noted above, this implies that $(v(n))_{n\ge0}$ is previsible with respect to the filtration $(\F_n')_{n\ge0}$. Let $K=K(x)$ and $\eps=\eps(x)$ be as in Lemma~\ref{lem:deterministic} and let 
\[
 \tau = \inf\{n\in \N: 
 \exists~v\in\mathcal C_{v(n)} \mbox{ such that $v$ is }(\eps,K,\mathrm{X})-\mathrm{steep}\}.
\]
Then $\tau \le \tau_x$, by Lemma~\ref{lem:deterministic}. Moreover, 
by Lemma~\ref{lem:independence}, 
\begin{align*}
& \P\left(\exists~v\in\mathcal{R}_n\setminus \mathcal{R}_{n-1}~:~v\mbox{ is }\mbox{$(\eps,K,\mathrm{X})$-steep}~|~ \F_{n-1}'\right)\\
 = \ &\P\left(\exists~v\in\mathcal{R}_n\setminus \mathcal{R}_{n-1}~:~v\mbox{ is }\mbox{$(\eps,K,\mathrm{X}')$-steep}~|~ \F_{n-1}'\right)\\
 \le \ &
\P\left(\exists~v\in\mathcal{C}_{v(n)}~:~v\mbox{ is }\mbox{$(\eps,K,\mathrm{X}')$-steep}~|~ \F_{n-1}'\right)\\
\le \ &
\sup_{w\in\T_N} 
\P\left(\exists~v\in\mathcal{C}_w~:~v\mbox{ is }\mbox{$(\eps,K,\mathrm{X}')$-steep}\right).
\end{align*}
The first inequality uses the fact that $\mathcal{R}_n\setminus \mathcal{R}_{n-1} \subset \mathcal{C}_{v(n)}$ and the second inequality uses the independence of $\rX'$ and $\F_{n-1}'$.
Since $\rX'$ and $\rX$ have the same law, 
by Lemma~\ref{lem:subproblem}, with $\gamma=\gamma(K,\eps)$ as in that lemma the preceding bound yields that 
\[
\P\left(\exists~v\in\mathcal{R}_n\setminus \mathcal{R}_{n-1}~:~v\mbox{ is }\mbox{$(\eps,K,\mathrm{X})$-steep}~|~ \F_{n-1}'\right) \le  e^{-\gamma N}.
\]
We thus obtain 
\begin{align*}
\P\left(\tau=n\right)
& = \E\left(\Ind_{\tau > n-1}\cdot \P\left(\tau=n~|~\F_{n-1}'\right)\right) \\
& = \E\left(\Ind_{\tau > n-1}\cdot  \P\left(\exists~v\in\mathcal{R}_n\setminus \mathcal{R}_{n-1}~:~v\mbox{ is }\mbox{$(\eps,K,\mathrm{X})$-steep}~|~ \F_{n-1}'\right)\right) \\
& \le \P\left(\tau > n-1\right) \cdot e^{-\gamma N} \, ,
\end{align*}
from which the result is immediate. 
%
\end{proof}

\subsection{Proofs of Lemmas~\ref{lem:deterministic} and \ref{lem:subproblem}}
\label{sec:proofs}

\begin{proof}[Proof of Lemma~\ref{lem:deterministic}]
We first show that
\begin{equation}
\label{eq:limit_A}
 \lim_{K\to\infty} \sum_{k=1}^K \sqrt{\frac{A(k/K)-A((k-1)/K)}{K}} = \int_0^1 \sqrt{a(s)}\,ds.
\end{equation}
Indeed, for every $K\in\N$, the sum on the left-hand side is bounded from below by the integral on the right-hand side which is easily seen by applying Jensen's inequality to each term in the sum. As for the upper bound (which is the one we will need later), we have for every $K\in\N$,
\[
 \sum_{k=1}^K \sqrt{\frac{A(k/K)-A((k-1)/K)}{K}} \le \frac 1 K \sum_{k=1}^K \max_{x\in \left[\frac{k-1}K,\frac k K\right]} \sqrt{a(s)}\,ds,
\]
and this sum converges to the right-hand side of \eqref{eq:limit_A} by the Riemann-integrability of $\sqrt a$.

Now fix $x> x_* = \sqrt{2\log 2} \int_0^1 \sqrt{a(s)}\,ds$ and let $\eps>0$ be such that $ x_*\cdot\sqrt{1+\eps} < x$. Then by \eqref{eq:limit_A}, there exists $K\in\N$, such that
\[
 \sum_{k=1}^K \sqrt{(1 + \eps)\frac{(2\log 2) (A(k/K)-A((k-1)/K))}{K}} < x.
\]
Now fix $w \in V_N$. Then 
\[
X_w = \sum_{k=1}^K (X_{w[\lfloor N k/K\rfloor]}-X_{w[\lfloor N (k-1)/K\rfloor]})\, ,
\]
so if $w$ has no $(\eps,K,\rX)$-steep ancestor then 
\[
X_w < N \cdot \sum_{k=1}^K \sqrt{(1 + \eps)\frac{(2\log 2) (A(k/K)-A((k-1)/K))}{K}} < xN.
\]
This proves the lemma. 
\end{proof}

\begin{proof}[Proof of Lemma~\ref{lem:subproblem}]
First note that if $w'$ is an ancestor of $w$ then $\mathcal C_{w'} \subset \mathcal{C}_w$. It follows by subadditivity of probabilities that in proving the lemma we may restrict our attention to $w \in V_N$. For such $w$, we next decompose the probability we aim to bound spindle-wise: 
\begin{align*}
\P\left(\exists~v\in\mathcal C_{w}~:~v\mbox{ is }\mbox{$(\eps,K,\mathrm{X})$-steep}\right) 
& = \sum_{k=1}^K\P\left(\exists~v\in\mathcal C_{w,k}:v\mbox{ is }\mbox{$(\eps,K,\mathrm{X})$-steep}\right) \\
& = \sum_{k=1}^K \P\left(\exists~v\in\mathcal C_{w,k}:|v|=\lfloor Nk/K\rfloor,\mbox{ $v$ is }\mbox{$(\eps,K,\mathrm{X})$-steep}\right) 
\end{align*}
The number of nodes of $\mathcal{C}_{w,k}$ in generation 
$\lfloor Nk/K\rfloor$ is $2^{\lfloor Nk/K\rfloor - \lfloor N(k-1)/K\rfloor} \le 2^{1+N/K}$, so by symmetry, writing $w^{(k)} = w[\lfloor Nk/K\rfloor]$, we have
\begin{align*}
& \P\left(\exists~v\in\mathcal C_{w,k}:|v|=\lfloor Nk/K\rfloor,\mbox{ $v$ is }\mbox{$(\eps,K,\mathrm{X})$-steep}\right) \\
& = 2^{1+N/K}\cdot 
\P\left(X_{w^{(k)}} - X_{w^{(k-1)}} \ge N\cdot \sqrt{(1 + \eps)\frac{(2\log 2) (A(k/K)-A((k-1)/K))}{K}} \right).
\end{align*}
Finally, $X_{w^{(k)}} - X_{w^{(k-1)}}$ is a centred Gaussian with variance bounded by $(A(k/K)-A((k-1)/K))N$. By standard Gaussian tail estimates (see, e.g., \cite[Lemma 12.9]{MR2604525}) it follows that 
\begin{align*}
\P\left(X_{w^{(k)}} - X_{w^{(k-1)}} \ge N\cdot \sqrt{(1 + \eps)\frac{(2\log 2) (A(k/K)-A((k-1)/K))}{K}} \right)
& \le e^{-(1+\eps)(\log 2) N/K}\, ,
\end{align*}
and combining the three preceding displays yields that 
\[
\P\left(\exists~v\in\mathcal C_{w}~:~v\mbox{ is }\mbox{$(\eps,K,\mathrm{X})$-steep}\right) 
= \sum_{k=1}^K 2^{1+N/K} e^{-(1+\eps)(\log 2) N/K} \le 2K\cdot 2^{-\eps N/K}\, ,
\]
from which the lemma is immediate. 
\end{proof}

\section{A weaker algorithmic model}  \label{sec:weakermodel}

The CREM is the Gaussian process $(X_v)_{v \in V_N}$ sitting on the leaves of $\mathbb{T}_N$. As such, it's natural to ask what can be achieved algorithmically if one only has access to information about values on $V_N$ rather than on all of $\mathbb{T}_N$. This restricts the class of allowed algorithms, so for $x > x_*$ it clearly remains hard to find nodes $v \in V_N$ with $X_v > x_* N$. 

When $x < x_*$, we may modify the greedy algorithm to obtain an algorithm which runs in time $N^{2+o(1)}$ but only queries nodes in $V_N$. The modification works as follows. Fix a parameter $\ell \ge 1$ and write $m=2^\ell$. When we wish to sample an internal node $v\in V_g$, we instead estimate $X_v$ by averaging the values observed at a subset of its descendants in $V_N$. 

Let $v_1,\ldots,v_{m}$ be the generation-$(g+\ell)$-descendants of $v$. For each $1 \le i \le m$ let $\hat{v}_i$ be a  descendant of $v_i$ at generation $N$ chosen according to some arbitrary deterministic or random rule (but independently of the process $(X_v)_{v \in V_N}$). Then use $\hat{X}_v = \frac{1}{m} \sum_{i \le m} X_{\hat{v}_i}$ as a proxy for $X_v$ in the algorithm.
For $1 \le j \le \ell$, there are $2^j$ edges between generation-$(g+j-1)$ and generation-$(g+j)$ descendants of $v$; list the displacements along those edges as $D^j_1,\ldots,D^j_{2^j}$. These random variables are centred Gaussians with variance $N(A((g+j)/N-A((g+j-1)/N))$. They are mutually independent, and independent for different generations. Each random variable $X_{\hat{v}_i}$ may be decomposed as 
\[
X_{\hat{v}_i} = X_v + \sum_{j=1}^\ell D^j_{q(i,j)} + \Delta_{\hat{v}_i},
\]
for some indices $(q(i,j))_{1 \le j \le \ell}$. The random variables $\Delta_{\hat{v}_i}$ are independent of each other and of the $D^j_k$, and are centred Gaussians with variance $N(1-A((k+\ell)/N)) \le N$. 

Each variable $D^j_1,\ldots,D^j_{2^j}$ contributes to $2^{\ell-j}=m/2^j$ of the samples $X_{\hat{v}_i}$. It follows that 
\begin{align*}
\hat{X}_v & = X_v + \frac{1}{m} \sum_{j=1}^\ell \frac{m}{2^j} \sum_{k=1}^{2^j} D^j_k + \frac{1}{m} \sum_{i=1}^m \Delta_{\hat{v}_i}\, ,
\end{align*}
so, the sample error $E_v = \hat{X}_v-X_v$ 
is a centred Gaussian with variance 
\begin{align*}
& \sum_{j=1}^\ell 2^{-j} N(A((g+j)/N-A((g+j-1)/N)) + \frac{1}{m} N(1-A((k+\ell)/N))  \le \sup_{t \in [0,1]} a(t) + \frac{N}{2^\ell}\, .
\end{align*}
Taking $\ell = \lfloor \log_2 N\rfloor$, and recalling that $\sup_{t \in [0,1]} a(t) < \infty$, as noted just after the statement of Theorem~\ref{th:complexity_CREM}, it follows that the overall sample error is a Gaussian with variance bounded by some $B  \in (1,\infty)$. 

In a given step $n$ of the modified algorithm, we evaluate
\[
v(M,n+1) = \argmax_{\substack{v\ge v(M,n)\\|v| = (Mn+M) \wedge N}} \hat{X}_v= \argmax_{\substack{v\ge v(M,n)\\|v| = (Mn+M) \wedge N}} (X_v+E_v),
\]
using the sample means given by the above procedure. 
To bound the estimation error incurred in step $n$, note that each $E_v$ satisfies $\P(|E_v| \ge 2B \sqrt {\log N}) \le N^{-2}$. 
A union bound over the $N/M$ levels and over the $2^M$ sample errors in each level gives that the worst overall sample error is $O(\sqrt{\log N})$ with probability $1-O(2^M/(MN))$. 

Now take $M = \lfloor (\log N)^{2/3} \rfloor$ (the exponent $2/3$ can be replaced by any value in the interval $(1/2,1)$). Then with probability $1-o(1)$, the total accumulated error along the path built by the algorithm is $O((N/M) \cdot (\log N)^{1/2}) = o(N)$. The number of sample estimates computed by the algorithm is $O((N/M)\cdot 2^M)$. Each sample estimate involves $2^\ell = O(N)$ queries, so the total number of queries is $O(N^2 2^{(\log N)^{2/3}}) = N^{2+o(1)}$.

\section{Open questions}\label{sec:open_questions}
In this section we state a few questions/directions for future research.
\begin{enumerate}
\item Say that a sequence of measures $\nu_N$ on $V_N$ \emph{approximates} another sequence of measures $\mu_N$ on $V_N$ if 
\[
\frac 1 N D_{KL}(\nu_N\|\mu_N) \to 0,\quad N\to\infty,
\]
where $D_{KL}(\nu\|\mu) = \sum \nu(x)\log(\nu(x)/\mu(x))$ is the Kullback--Leibler divergence from $\mu$ to $\nu$. For a given $\beta\ge 0$, say that the Gibbs measure $\mu_\beta = \mu_{\beta,N}$ of the CREM can be \emph{efficiently approximated} if there exists a sequence of algorithms $\mathrm{v}_N$ on $\T_N$ and a sequence of stopping times $\tau_N$, such that with high probability, the law of $v_N(\tau_N)$ approximates $\mu_{\beta,N}$ and $\tau_N \le N^{O(1)}$. We then conjecture: \emph{there exists a threshold $\beta_G$, such that $\mu_\beta$ can be efficiently approximated if $\beta < \beta_G$ and cannot if $\beta > \beta_G$.} We denote the corresponding conjectural threshold in the negative energies by $x_G$.

Let us give some details and ideas supporting this conjecture. We use the notation from Appendix~\ref{sec:app}. Consider for every $x\ge0$ the variational problem
\[
b_x = \operatorname{argmin}\left\{E(b,1): b\in\mathcal R,\ \int_0^1 b(t)\,dt = x\right\}.
\]
We expect that typical particles sampled from the Gibbs measure $\mu_\beta$ with inverse temperature $\beta$ corresponding to $x$ follow approximately the trajectory with derivative $b_x$. It is known \cite{Bovier2004a} that $x$ and $\beta$ are in correspondence through
\begin{align*}
x &= \sqrt{2\log 2} \int_0^{t_0} \sqrt{\hat a(s)}\,ds + \beta (1-A(t_0))\\
\text{where}\quad t_0 &= t_0(\beta) = \sup\left\{t\ge0:\beta > \frac{\sqrt{2\log 2}}{\sqrt{\hat a(t)}}\right\}\, .
\end{align*}
Furthermore, one can prove, along the lines of Proposition~\ref{prop:r_max} (see also \cite{Bovier2004a}), that
\begin{align*}
b_x(t) &= 
\begin{cases} 
v(t), & t\le t_0\\
\beta a(t), & t> t_0\, .
\end{cases}
\end{align*}
From this, it is not too complicated to derive the following fact: \emph{$E(b_x,t)$ is non-increasing in~$t$ if and only if $\beta \le \beta_G$}, where
\[
\beta_G = \frac{\sqrt{2\log 2}}{\operatorname{ess~sup}_{t\ge t_G} \sqrt{a(t)}},\quad t_G = \sup\{t\in[0,1]: A(t) = \hat A(t)\},
\]
with $\beta_G = \sqrt{2\log 2}/\sqrt{\hat a(1)}$ in the special case $A = \hat A$,
or equivalently, \emph{if and only if $x\le x_G$}, where
\[
x_G = \sqrt{2\log 2} \int_0^{t_0(\beta_G)} \sqrt{\hat a(s)}\,ds + \beta_G (1-A(t_0(\beta_G))).
\]

We now argue that these are the values of the thresholds mentioned above (and so, {\em a fortiori}, that the thresholds exist). First, if $x < x_G$, then $E(b_x,t)$ is non-increasing in $t$, and one can use a variant of the linear-time algorithm from this article in order to sample particles approximately following the trajectory with derivative $b_x$: at each step, instead of choosing the maximum vertex among the $2^M$ descendants $M$ levels down the tree one samples a vertex according to an appropriate Gibbs measure restricted to these descendants. This Gibbs measure is chosen so that particles sampled from it have the right ``speed'', which we believe is enough to yield that the output is a good approximation to $\mu_{\beta,N}$ with high probability; the fact that the energy is non-increasing precisely ensures that such particles are very likely to exist. 

On the other hand, if $x>x_G$, i.e.~if the energy $E(b_x,t)$ has a point of increase, then an argument similar to the proof of the hardness part of Theorem~\ref{th:complexity_CREM} should allow one to show that approximating the Gibbs measure is as hard as finding a subset of the tree which satisfies an event of exponentially small probability, and thus needs a time which is exponential in $N$.

One can check that $\beta_G < \beta_c$ if $A(x) < x$ for all $x\in(0,1)$ (in this case, $t_G = 0$), and that $\beta_G > \beta_c$ if $A$ is strictly concave (in this case, $t_G = 1$).

\item For fixed $A$, and $x > x_*(A)$, let 
\[ 
\gamma_*(x) = \limsup_{N\to\infty}\left[ \inf\left\{\gamma>0: \exists \text{~sequence~of~algorithms~$\mathrm{v}_N$~on~}\mathbb{T}_N: \P\left(\tau_x(\mathrm{v}_N)\le e^{\gamma N}\right) \stackrel{N\to\infty}\to 1\right\}\right].
\]
What is the behaviour of $\gamma_*$? In particular, what is the right-derivative of $\gamma_*(x)$ at $x=x_*$? Assuming that $a$ is piecewise continuous, we expect that 
\[
\gamma_*'(x_*) = \frac{\sqrt{2\log 2}}{\max_{0\le s\le t\le 1} \left\{ \sqrt{a(t)}-\sqrt{a(s)}\right\}}.
\]
The reason is that the best strategy in order to get to $x = x_* + h$, with $h$ small, should be to follow the natural speed path most of the time except at two times $s\le t$ such that $a(s)$ is small and $a(t)$ is large: at time $s$, the path should slow down in order to gain entropy (an exponential number of particles following this path), and this entropy should be used up at a later time $t$ with an atypically fast trajectory.

More generally, we expect that $\gamma_*$ is given in terms of the following variational principle (this was suggested to us by Bastien Mallein):
\begin{equation}
\label{eq:bastien}
\gamma_*(x) = \inf\left\{\gamma > 0: \sup\left\{\int_0^1 b(t)dt : b \in \mathcal{R},\,\forall t\in[0,1]:E(b,t) \ge - \gamma\right\} > x\right\}.
\end{equation}
The rationale behind the constraint that $\forall t\in[0,1]:E(b,t) \ge - \gamma$ is that this ensures that throughout the whole process, the number of particles following such a trajectory until a certain time is bounded by $\exp((\gamma+o(1)) N)$, and so this is the maximum size of the ``pool'' of particles from which we have to find one leading up to $xN$.

Proving an upper bound on $\gamma_*$ using the above heuristics should be achievable, but we expect that proving the corresponding lower bound might be hard, unless one restricts the class of admissible algorithms.

\item Related to the last question, it is natural to ask about the form of the transition from linear to exponential complexity as the threshold $x_*$ is crossed. In particular, does any algorithm which finds a node $v \in V_N$ with $X_v \ge x_* N -o(N)$ require $\omega(N)$ queries with high probability? We expect this is the case. A related question for a homogeneous branching random walk has been studied by Pemantle \cite{Pemantle2009}.

\item \label{item:spherical_threshold}
 It is clearly of interest to determine algorithmic barriers for other models. The {\em spherical $p$-spin} models, defined in Section~\ref{sec:discussion}, are one setting where progress might be achievable. For these models, typical local minima of the energy landscape concentrate near energy $-N\cdot E_{\infty}$ for $N$ large, where $E_{\infty}=2\sqrt{(p-1)/p}$. Moreover, with high probability, there are no critical points of finite index at higher energies \cite{Auffinger2013a}. As such, it is reasonable to guess that $E_{\infty}$ identifies the level at which algorithms based on gradient descent will be  blocked. It would be very interesting to know whether $E_{\infty}$ in fact determines the algorithmic threshold. 

\end{enumerate}

\appendix

\section{The maximizer of the variational principle}\label{sec:app}
In this appendix, we detail some of the assertions from Section~\ref{sec:howtounderstand}. We use throughout the notation defined in the introduction.

We work with the derivatives of the functions, rather than the functions themselves. Recall the definition of the functional $E$ from Section~\ref{sec:howtounderstand}, and let
\[
 \mathcal R = \{b:[0,1] \to \R\mbox{ measurable}: \forall t\in[0,1]: E(b,t) \le 0\}.
\]
Also recall that $\hat{A}$ denotes the concave hull of $A$, and $\hat{a}$ the left-derivative of~$\hat{A}$. 
Define a function $v:[0,1] \to [0,\infty)$ by setting 
\[
v(s) = a(s) \cdot \Big(\frac{2\log 2}{\hat{a}(s)}\Big)^{1/2}\, .
\]
\begin{proposition}\label{prop:r_max}
Under the conditions of Theorem~\ref{th:complexity_CREM}, the function $v$ is an element of $\mathcal{R}$, and 
\[
\int_0^1 v(t)\mathrm{d}t = \sup\left\{\int_0^1 b(t)dt : b \in \mathcal{R}\right\} = (2\log 2)^{1/2} \int_0^1 \hat{a}(t)^{1/2}\mathrm{d}t\, .
\]
\end{proposition}
In proving the proposition, we shall use the following fairly straightforward fact about the concave hull. 
\begin{lemma}\label{lem:flip}
For any measurable function $g:[0,\infty) \to [0,\infty)$, 
\and any $t \in (0,1]$, 
\[
\int_0^t g(\hat{a}(x)) A(\mathrm{d}x) + g(\hat{a}(t))(\hat{A}(t)-A(t))=\int_0^{t} g(\hat{a}(x)) \hat{a}(x)\mathrm{d}x.
\]
\end{lemma}
\begin{proof}
Remark that the set $\mathcal{I} = \{x: \hat{A}(x)\ne A(x)\}$ may be expressed as a countable union of intervals, $\mathcal{I} = \bigcup_{i \in I} (x_i,y_i)$, so that $\hat{a}$ is constant on each interval and takes distinct values on any two intervals. 


For $t \in [0,1]\setminus \mathcal{I}$, we have 
\begin{align*}
\int_{0}^t g(\hat{a}(x)) \hat{A}(\mathrm{d}x) 
& = \int_{\mathcal{I} \cap [0,t]} g(\hat{a}(x)) \hat{A}(\mathrm{d}x)
+ \int_{[0,t]\setminus \mathcal{I}} g(\hat{a}(x)) A(\mathrm{d}x),
\end{align*}
because $A(x)=\hat{A}(x)$ for all $x \in [0,t]\setminus \mathcal{I}$ and $A$ is absolutely continuous. We may write 
$\mathcal{I} \cap [0,t] = \bigcup_{j \in J} (x_j,y_j)$ for some $J \subset I$. For each $j \in J$, 
the function $\hat{a}$ is constant on $(x_j,y_j]$ (recall that $\hat a$ is the left-derivative of $\hat A$), so 
\begin{align*}
\int_{\mathcal{I} \cap [0,t]} g(\hat{a}(x)) \hat{A}(\mathrm{d}x)
& = \sum_{j \in J} \int_{(x_j,y_j)} g(\hat{a}(x)) \hat{A}(\mathrm{d}x) \\
& = \sum_{j \in J} g(\hat{a}(y_j))(\hat{A}(y_j)-\hat{A}(x_j))\\
& = \sum_{j \in J} g(\hat{a}(y_j))(A(y_j)-A(x_j))\\
& = \int_{\mathcal{I} \cap [0,t]} g(\hat{a}(x)) A(\mathrm{d}x)\, .
\end{align*}
Combining the two preceding displays gives 
\[
\int_{0}^t g(\hat{a}(x)) \hat{A}(\mathrm{d}x) = 
\int_{0}^t g(\hat{a}(x)) A(\mathrm{d}x) = \int_{0}^t g(\hat{a}(x)) A(\mathrm{d}x) + g(\hat{a}(t))(\hat{A}(t)-A(t)),\]
the second equality since $\hat{A}(t)=A(t)$.


Next, if $t \in \mathcal{I}$ then $t \in (x_i,y_i)$ for some $i \in I$. In this case, as $\hat{a}$ is constant on $(x_i,y_i]$,
\[ 
\int_{x_i}^t g(\hat{a}(x)) \hat{A}(\mathrm{d}x)
= g(\hat{a}(t))(\hat{A}(t)-\hat{A}(x_i)) 
= g(\hat{a}(t)) (\hat{A}(t)-A(t))
+
\int_{x_i}^t g(\hat{a}(x)) A(\mathrm{d}x)\, ,
\] 
the final equality since $\hat{A}(x_i)=A(x_i)$. Since $x_i \in [0,1]\setminus \mathcal{I}$, we also have $\int_{0}^{x_i} g(\hat{a}(x)) \hat{A}(\mathrm{d}x) = \int_{0}^{x_i} g(\hat{a}(x)) A(\mathrm{d}x)$, which combined with the preceding display establishes the identity in this case as well.
\end{proof}

\begin{proof}[Proof of Proposition~\ref{prop:r_max}]
First, by Lemma~\ref{lem:flip} applied with  $g(y)=y^{-1}$, for all $t \in [0,1]$ we have 
\[
\int_0^t \frac{v(s)^2}{2a(s)} dx 
= \log 2 \int_0^t \frac{1}{\hat{a}(s)} a(s)ds =(\log 2) \left(t-\frac{\hat{A}(t)-A(t)}{\hat{a}(t)} \right) \le t \log 2\, ,
\]
so $v \in \mathcal{R}$. 

Next, following Mallein~\cite{Mallein2015}, we introduce the family of log-Laplace transforms corresponding to the displacements in $\T_N$. For $t \in (0,1)$, let $Z_t$ be a centred Gaussian with variance $a(t)$; then 
\[
\kappa_t(\theta) = \log \left(2\E{e^{\theta Z_t}}\right) = \log (2 e^{a(t) \theta^2/2}) = \log 2 + \frac{a(t)\theta^2}{2}\, .
\]
Then let 
\[
\kappa_t^*(x) 
= \sup_{\theta \in\R}(\theta x - \kappa_t(\theta))
=  \frac{x^2}{2a(t)}-\log 2\,
\]
be the Fenchel-Legendre dual of $\kappa_t$. Note that for all $t \in [0,1]$, 
\begin{equation}
\label{eq:Ekappa}
E(b,t) = \int_0^t \kappa_s^*(b(s)) ds. 
\end{equation}
Finally, define for $t \in [0,1]$,
\[
\theta_t \coloneqq \frac{d}{dx} \kappa_t^*(x)\Big|_{x=v(t)} = \frac{v(t)}{a(t)} = \left(\frac {2\log 2}{\hat{a}(t)}\right)^{1/2}\, .
\]

Now fix any function $b \in \mathcal{R}$. 
We wish to show that $\int_0^1 v(t) \mathrm{d}t \ge \int_0^1 b(t) \mathrm{d}t$.
Since $\kappa_t^*$ is convex, it follows that
\[
\kappa_t^*(b(t)) \ge \kappa_t^*(v(t)) + \theta_t(b(t)-v(t)). 
\]
Integrating over $t$ now yields that 
\begin{equation}\label{eq:tobound}
\int_0^1 (v(t)-b(t)) dt \ge \int_0^1 \frac{\kappa_t^*(v(t)) - \kappa_t^*(b(t))}{\theta_t}dt
\end{equation}

To bound the right-hand side, first note that 
\begin{align}
\int_0^1 \frac{\kappa_t^*(v(t))}{\theta_t} 
& = \int_0^1 \left(\frac{v(t)}{2} - \frac{\log 2}{\theta_t}\right) \nonumber\\
& = \left(\frac{\log 2}{2}\right)^{1/2} \int_0^1 \frac{1}{\hat{a}(t)^{1/2}}\big(a(t)-\hat{a}(t)\big) dt \nonumber\\
& = 0\, ,
\label{eq:feeltheflip}
\end{align}
the final equality by Lemma~\ref{lem:flip} applied with $g(y)=y^{-1/2}$. 
Next, using \eqref{eq:Ekappa} and integration by parts, we have 
\begin{align*}
\int_0^1 \frac{\kappa_t^*(b(t))}{\theta_t}dt
& 
= \frac{E(b,1)}{\theta_1} + \int_0^1 \frac{E(b,t)}{\theta_t^2} \theta'_t dt\, .
\end{align*}
Since $b \in \mathcal{R}$ we have $E(b,t) \le 0$ for all $t \in[0,1]$. Also, $\hat{a}(t)$ is non-increasing so $\theta_t$ is non-decreasing; it follows that $\int_0^1 \frac{\kappa_t^*(b(t))}{\theta_t}dt \le 0$. 
Combining this bound with (\ref{eq:tobound}) and (\ref{eq:feeltheflip}) yields that 
\[
\int_0^1 (v(t)-b(t))dt  \ge -\int_0^1  \frac{\kappa_t^*(b(t))}{\theta_t}dt \ge 0\, .
\] 
This shows $v$ has the largest integral of any element of $\mathcal{R}$. Finally, the fact that 
\[
\int_0^1 v(t) \mathrm{d}t = (2\log 2)^{1/2} \int_0^1 \hat{a}(t)^{1/2}\mathrm{d}t
\]
is again a consequence of Lemma~\ref{lem:flip} applied with $g(y)=y^{-1/2}$. 
\end{proof}

\bibliography{2018_complexity_CREM,louigi}

\end{document}